\documentclass{amsart}
\usepackage{amsmath, amssymb}
\begin{document}

\newtheorem{theorem}{Theorem}[section]
\newtheorem{lemma}[theorem]{Lemma}
\newtheorem{corollary}[theorem]{Corollary}
\newtheorem{fact}[theorem]{Fact}
\newtheorem{proposition}[theorem]{Proposition}
\newtheorem{claim}[theorem]{Claim}

\theoremstyle{definition}
\newtheorem{example}[theorem]{Example}
\newtheorem{remark}[theorem]{Remark}
\newtheorem{definition}[theorem]{Definition}
\newtheorem{question}[theorem]{Question}
\newtheorem{conjecture}[theorem]{Conjecture}

\def\ki{K\{X\}_{\operatorname{int}}}
\def\kzi{K[Z]_{\operatorname{int}}}
\def\kzhi{K\left[Z_i:i<h\right]_{\operatorname{int}}}
\def\k{K\{X\}}
\def\kz{K[Z]}
\def\kzh{K\left[Z_i:i<h\right]}
\def\krzhr{K_r\left[Z_i:i<h_r\right]}
\def\krzh{K_r\left[Z_i:i<h\right]}
\def\rank{\operatorname{rank}}
\def\int{\operatorname{int}}
\def\ord{\operatorname{ord}}

\title[Nonstandard methods for bounds in differential polynomial rings]{Nonstandard methods for bounds in\\ differential polynomial rings}
\author{Matthew Harrison-Trainor}
\author{Jack Klys}
\thanks{Matthew Harrison-Trainor and Jack Klys were partially supported by an NSERC USRA.}
\author{Rahim Moosa}
\thanks{Rahim Moosa was partially supported by an NSERC Discovery Grant}
\address{Department of Pure Mathematics\\
University of Waterloo\\
Waterloo, Ontario\\
Canada N2L 3G1}

\date{May 2nd, 2011.}

\subjclass[2000]{Primary 12H05. Secondary 03H05}

\begin{abstract}
Motivated by the problem of the existence of bounds on degrees and orders in checking primality of radical (partial) differential ideals,
the nonstandard methods of van den Dries and Schmidt [``Bounds in the theory of polynomial rings over fields. A nonstandard approach.", Inventionnes Mathematicae, 76:77--91, 1984] are here extended to differential polynomial rings over differential fields.
Among the standard consequences of this work are:  a partial answer to the primality problem, the equivalence of this problem with several others related to the Ritt problem,
and the existence of bounds for characteristic sets of minimal prime differential ideals and for the differential Nullstellensatz.
\end{abstract}

\maketitle

\section{Introduction}

\noindent
This paper is concerned with existence proofs of bounds in the theory of differential polynomial rings over differential fields.
We are motivated by the following open question in differential-algebraic geometry: Given a differential-algebraic family of Kolchin-closed sets $\{V_a\subseteq L^n:a\in L^m\}$, where $L$ is a differentially closed field of characteristic zero in several commuting derivations, is the set
$$D:=\{a\in L^m:V_a\text{ is irreducible}\}$$
Kolchin-constructible?
Because of quantifier-elimination for differentially closed fields, this is equivalent to asking whether $D$ can be defined by a first-order formula in the language of differential rings.
Now if $V_a$ is defined by the vanishing of the finite system $S_a$ of differential polynomials over $L$, then $V_a$ is irreducible if and only if the radical differential ideal generated by $S_a$, denoted by $\{S_a\}$, is prime.
Membership of $f_b$ in $\{S_a\}$, for a given differential polynomial with coefficients $b$, is a first-order property of $(a,b)$ because it is equivalent to the vanishing of $f_b$ on~$V_a$.
Primality of $\{S_a\}$, on the other hand, is not on the face of it first-order in $a$ because it requires quantification over all (coefficients of) differential polynomials $f$ and $g$ such that $fg\in\{S_a\}$.
If, however, we could restrict these quantifications to differential polynomials of bounded degree and order -- bounded independently of~$a$ -- then we would have a first-order definition for $D$.
In fact, one expects to have bounds that are also independent of the base differential field.

\begin{conjecture}[Definability of primality]
\label{primeconjecture}
For every $d$ there exists $r=r(d,n,m)$
such that for every field of characteristic zero equipped with $m$ commuting derivations $\Delta=\{\delta_1,\dots,\delta_m\}$ and every finite set $S$ of $\Delta$-polynomials in $n$ indeterminates $X=(X_1\dots,X_n)$ over $k$ of degree and order $\leq d$, the following are equivalent:
\begin{itemize}
\item[(i)]
$\{S\}$, the radical $\Delta$-ideal generated by $S$, is prime,
\item[(ii)]
 $\{S\}$ is proper and for all $\Delta$-polynomials $f$ and $g$ in $X$ over $k$ of degree and order $\leq r$, if $fg\in \{S\}$ then $f\in\{S\}$ or $g\in\{S\}$.
\end{itemize}
\end{conjecture}

The purely algebraic analogue of Conjecture~\ref{primeconjecture}, which asks for a bound on degrees for checking the primality of radical ideals in polynomial rings over fields, has an affirmative solution, and this is why irreducibility in algebraic families of Zariski-closed sets is definable (in the language of rings).
Constructive proofs go back to the work of Hermann~\cite{hermann26} and later Seidenberg~\cite{seidenberg74}.
Clearer and more efficient proofs of the {\em existence} of bounds were given by van den Dries and Schmidt in~\cite{vandenDries} using a nonstandard approach; that is, by studying ultraproducts of polynomial rings over fields.
Our goal here is to extend the methods of van den Dries and Schmidt to the differential setting; to study the ultraproducts of differential polynomial rings over differential fields and to bring that to bear on Conjecture~\ref{primeconjecture}.

We obtain a partial solution whereby we are able to bound one of the two quantifiers in the definition of primality; see Theorem~\ref{prop:Partial_Prime_Bound_Standard} below.
We are also able to prove the equivalence of Conjecture~\ref{primeconjecture} with four other natural existence-of-bounds conjectures in differential algebra; see Theorem~\ref{primeconjecture-equivalents}.
Here we are motivated by~\cite{RittProbAlgorithms} where a related series of problems around the existence of {\em algorithms} in computational differential algebra are shown to be equivalent.
As a byproduct of our analysis we also obtain quick existence proofs of bounds in two other areas where bounds have been obtained by (or are deducible from) constructive methods; namely for characteristic sets of minimal prime differential ideals (Theorem~\ref{pro:CharSetBound}) and for the differential Nullstellensatz (Theorem~\ref{thm:Effective_Nullstellensatz}).

The results described above are proved in the final two sections of the paper, based on the study of {\em internal} differential polynomials carried out in Section~\ref{k-ki-sect}.
In Sections~\ref{revsect} and~\ref{intdelt} below we review the relevant notions from differential algebra and introduce the corresponding nonstandard setting.

We are grateful to Alexey Ovchinnikov for his explanations regarding several points of constructive differential algebra.

\bigskip
\section{Differential-algebraic preliminaries}
\label{revsect}

\noindent
We begin with a quick review of differential polynomial rings, primarily to set notation.
See~\cite{Kolchin} for details.

Suppose $m,n<\omega$, $\big(k,\Delta=\{\delta_1,\dots,\delta_m\}\big)$ is a differential field of characteristic zero in $m$ commuting derivations, and $X=(X_1,\dots,X_n)$ is an $n$-tuple of indeterminates.
Let $\Theta=\left\{ \delta_{1}^{e_{1}}\cdots\delta_{m}^{e_{m}}:e_{i}\geq0\right\} $
be the set of $\Delta$-operators and set $\Theta X:=\left\{ \theta X_{i}:\theta\in\Theta,1\leq i\leq n\right\}$.
The {\em ring of $\Delta$-polynomials in $X$ over $k$}, denoted by $k\{X\}$, is the ring of (commutative) polynomials in $\Theta X$ over $k$ equipped with the natural structure of a $\Delta$-ring.
For this reason the $X_i$ are called {\em differential indeterminates} while the elements of $\Theta X$ are the {\em algebraic indeterminates}.
For a subset $S\subset k\left\{ X\right\} $, we use $\left(S\right)$,
$\left[S\right]$, and $\left\{ S\right\} $ to denote, respectively,
the ideal, $\Delta$-ideal, and radical $\Delta$-ideal generated
by $S$.
By the {\em degree} of a $\Delta$-polynomial $f\in k\{X\}$ we will mean the total degree of $f$ as an element of $k[\Theta X]$, and by its {\em order} we mean the maximum order of the $\Delta$-operators that appear in $f$.
(The {\em order} of a $\Delta$-operator $\delta_{1}^{e_{1}}\cdots\delta_{m}^{e_{m}}$ is $e_1+\cdots+e_m$.)

We rank the algebraic indeterminates by 
$\displaystyle \delta_{1}^{e_{1}}\cdots\delta_{m}^{e_{m}}X_{i}<\delta_{1}^{f_{1}}\cdots\delta_{m}^{f_{m}}X_{j}$
if and only if
$\displaystyle \left(\sum_{l=1}^{m}e_{l},i,e_{1},\ldots,e_{m}\right)<\left(\sum_{l=1}^{m}f_{l},j,f_{1},\ldots,f_{m}\right)$
in the lexicographic ordering.
According to this ranking we enumerate $\Theta X$ as $\left(Z_h:h<\omega\right)$, and define the {\em height} of an algebraic indeterminate $v$ to be the $h$ such that $v=Z_h$. 
Given $f\in k\{X\}\setminus k$, the {\em leader} of $f$, denoted by $v_f$, is the highest ranking algebraic indeterminate appearing in $f$, and the {\em height} of $f$ is by definition the height of its leader.
The leading coefficient when $f$ is written as a polynomial in $v_f$ is called the {\em initial} of $f$ and is denoted by $I_f$.
Note that the leader of $I_f$ is strictly less than $v_f$ in rank.
The derivative of $f$ with respect to $v_f$ is called the {\em separant} of $f$ and will be denoted by $S_f$.
A consequence of these definitions is that for any $\theta\in\Theta\setminus\{1\}$, the initial of $\theta f$ is $S_f$.

The ranking of the algebraic indeterminates extends to $\Delta$-polynomials by the lexicographic ordering on the pair $\left(v_f, \deg_{v_f}(f)\right)$.
So for example both $\rank(I_f)<\rank(f)$ and $\rank(S_f)<\rank(f)$.
This ranking extends also to finite sequences of $\Delta$-polynomials by the following lexicographic-like ranking:
$\rank(f_1,\dots,f_r)<\rank(g_1,\dots,g_s)$ if either there is $k\leq\min(r,s)$ with $\rank(f_i)=\rank(g_i)$ for all $i<k$ and $\rank(f_k)<\rank(g_k)$, or $r>s$ and $\rank(f_i)=\rank(g_i)$ for all $i\leq s$.

Suppose $f,g\in k\{X\}$.
Recall that $f$ is {\em reduced with respect to} $g$ if $\theta v_g$ does not appear in $f$ for any $\theta\in\Theta\setminus\{1\}$ and the degree of $v_g$ in $f$ is strictly less than its degree in $g$.
In particular, if $\rank(f)<\rank(g)$ then $f$ is reduced with respect to $g$, but the converse need not hold; even when $f$ is reduced with respect to $g$ higher ranking algebraic indeterminates than $v_g$ may appear in $f$, just not any that can be obtained from $v_g$ by applying the derivations.
A set of $\Delta$-polynomials is called {\em autoreduced} if every element is reduced with respect to all the others.
Autoreduced sets are always finite, 
and we canonically view them as finite sequences by ordering the elements of an autoreduced set according to increasing height.
This induces a ranking on the autoreduced subsets of $k\left\{X\right\}$.

Given a subset $S\subseteq k\{X\}$, a {\em characteristic set} for $S$ is a lowest ranking autoreduced subset of $S$.
It is a basic fact that every $\Delta$-ideal has a characteristic set.
Moreover, prime $\Delta$-ideals are determined by their characteristic sets in the following sense: if $P$ is a prime $\Delta$-ideal and $\Lambda$ is a characteristic set for $P$ then 
$$P=I[\Lambda]:=\{g\in k\{X\}:H_\Lambda^tg\in [\Lambda], \text{ for some } t\in\mathbb N\}$$
where by definition $\displaystyle H_{\Lambda}:=\prod_{f\in\Lambda} S_fI_f$.
For notational reasons specific to this paper we are here denoting by $I[\Lambda]$ what is usually denoted in the literature by $[\Lambda]:H_{\Lambda}^\infty$.

Finally let us recall the characterisation of characteristic sets of prime $\Delta$-ideals.
First, given a finite set $\Lambda\subset\left\{X\right\}$ and $h<\omega$, $(\Lambda)_h$ denotes the ideal generated by $\Lambda$ together with all the derivatives of $\Lambda$ that are of height at most $h$.
An autoreduced set $\Lambda$ is {\em coherent} if whenever $f\neq g$ in $\Lambda$ are such that $\theta_fv_f=\theta_gv_g=Z_h$ for some $\theta_f,\theta_g\in\Theta$, and if $\theta_f,\theta_g$ are such that the height $h$ is minimal possible, then $S_g\theta_ff-S_f\theta_gg\in\left(\Lambda\right)_{h-1}$.

\begin{fact}[Lemma 2 of $\S$IV.9 of Kolchin~\cite{Kolchin}]
\label{thm:Kolchin}
Let $\Lambda$ be a finite subset of $k\left\{ X\right\}$.
Then $\Lambda$ is a characteristic set of a prime $\Delta$-ideal if and only if
\begin{itemize}
\item[(1)] $\Lambda$ is coherent, and
\item[(2)] $I(\Lambda):=\left\{g\in k\left\{X\right\}:H_\Lambda^tg\in (\Lambda),\text{ for some }t\in\mathbb N\right\}$ is a prime ideal containing no nonzero elements that are reduced with respect to $\Lambda$.
\end{itemize}
\end{fact}
\noindent
The main usefulness of this criterion rests in the fact that it makes reference only to the ideal $(\Lambda)$ and not the $\Delta$-ideal $[\Lambda]$.

\bigskip
\section{Internal $\Delta$-polynomials}
\label{intdelt}

\noindent
We will be using nonstandard methods just as they were used in~\cite{vandenDries}, a gentle introduction to which can be found in Section~3 of~\cite{NonstandardReference}.
Fix once and for all an index set $\mathbb{I}$ and a nonprincipal ultrafilter $\mathcal U$ on $\mathbb I$.
We work in the ultrapoduct with respect to $\mathcal U$ of whatever universe of algebraic objects we are interested in.
For any set $A$ we denote by $A^*$ its nonstandard interpretation,  namely $\displaystyle \prod_\mathcal{U} A$.
Given an element $a$ of the ultraproduct we will usually fix an $\mathbb I$-indexed sequence representing that element and use $a(r)$ to denote the $r$th co-ordinate of that representative.
We say that a condition $P$ is true of $a$ ``co-ordinatewise almost everywhere'' if the set of indices $r\in\mathbb I$ for which $P$ is true of $a(r)$ is a member of the ultrafilter.
This does not depend on the choice of representative for $a$.
Recall that a subset $S$ of an ultraproduct is {\em internal} if there exists a sequence of subsets $\big(S_r:r\in \mathbb I\big)$, called the {\em family of components} of $S$, such that $a\in S$ if and only if $a(r)\in S_r$ almost everywhere.
Note that every finite set is internal with components the set of co-ordinates, that is, $S_r=\{a(r):a\in S\}$ for all $r\in\mathbb I$.
We say that a condition $P$ is true of an internal set $S$ ``componentwise almost everywhere'' if the set of indices $r\in\mathbb I$ for which $P$ is true of $S_r$ is a member of the ultrafilter.
This too does not depend on the choice of family of components for $S$.

Suppose $\left\{ K_r:r\in\mathbb{I}\right\}$ is a sequence of $\Delta$-fields of characteristic zero.
Fixing differential indeterminates $X=(X_1,\dots,X_n)$, we have the $\Delta$-ring of {\em internal $\Delta$-polynomials over $K$}, namely $\displaystyle\prod_\mathcal{U}\left(K_r\{X\}\right)$, which, following~\cite{vandenDries} in the algebraic case, we denote by $\ki$ where $\displaystyle K:=\prod_\mathcal{U}K_r$.
Almost every differential algebraic notion has an internal analogue in $\ki$.
For example, given $f,g\in\ki$, we say that  {\em $f$ is of lower internal rank than} $g$ if it is of lower rank co-ordinatewise almost everywhere, that is, if $\big\{r\in\mathbb I:f(r)\text{ is of lower rank than }g(r)\big\}\in\mathcal U$.
Similarly, {\em $f$ is internally reduced with respect to $g$} if it is reduced with respect to $g$ co-ordinatewise almost everywhere.
An internal subset  $\Lambda\subset\ki$ is \emph{internally autoreduced} if it is componentwise autoreduced almost everywhere.
Note that unlike for standard autoreduced sets, internally autoreduced subsets of $\ki$ need not be finite (though they will be componentwise finite almost everywhere).
Nevertheless we can rank the internally autoreduced sets by declaring that $\rank\left(\Lambda\right)<\rank\left(\Gamma\right)$ if this is the case componentwise almost everywhere.
Similarly,
an {\em internally characteristic set} for an internal set $S$, is an internal subset $\Lambda\subset S$ such that $\Lambda_r$ is characteristic for $S_r$ for almost all $r$.
As an illustration of how these definitions play out, we prove the following straightforward equivalences.

\begin{lemma}
\label{sample}
Suppose $S,\Lambda\subset\ki$ are internal sets.
\begin{itemize}
\item[(i)]
$\Lambda$ is internally autoreduced if and only if every element of $\Lambda$ is internally reduced with respect to every other element of $\Lambda$.
\item[(ii)]
$\Lambda$ is an internally characteristic set for $S$ if and only if $\Lambda$ is a minimally ranked internally autoreduced subset of $S$.
\end{itemize}
\end{lemma}

\begin{proof}
For the left-to-right direction of~(i), suppose $V\in\mathcal U$ is such that $\Lambda_r$ is autoreduced in $K_r\left\{X\right\}$ for all $r\in V$.
Given $f\neq g$ in $\Lambda$, shrinking $V$ if necessary, we have that $f(r)\neq g(r)$ are in $\Lambda_r$ and hence are reduced with respect to each other.
It follows that $f$ and $g$ are internally reduced with respect to each other.
For the converse, suppose $W:=\{r\in\mathbb I:\Lambda_r\text{ is autoreduced}\}\notin\mathcal U$.
Then $V:=\mathbb I\setminus W\in \mathcal U$.
For each $r\in V$, let $f_r\neq g_r$ in $\Lambda_r$ be such that $f_r$ is not reduced with respect to $g_r$.
Let $f,g\in\ki$ be such that $f(r)=f_r$ and $g(r)=g_r$ for almost all $r$.
Then $f\neq g$ are in $\Lambda$ and $f$ is not internally reduced with respect to $g$.

Now suppose that $\Lambda$ is internally characteristic for $S$, and let $V\in\mathcal U$ be co-ordinates where $\Lambda_r$ is a characteristic subset of $S_r$.
In particular, $\Lambda$ is internally autoreduced.
Now if $\Sigma\subset S$ is any other internally autoreduced set, after possibly shrinking $V$, we have that $\Sigma_r\subset S_r$ is autoreduced and hence $\rank\left(\Lambda_r\right)\leq\rank\left(\Sigma_r\right)$, for all $r\in V$.
So $\rank\left(\Lambda\right)\leq\rank\left(\Sigma\right)$.
This proves that $\Lambda$ is minimally ranked among the internally autoreduced subset of $S$.
For the converse, suppose that $\Lambda\subset S$ is internally autoreduced but $W:=\{r\in\mathbb I:\Lambda_r\text{ is characteristic for }S_r\}\notin\mathcal U$.
Then $V:=\mathbb I\setminus W\in \mathcal U$.
Shrinking $V$ we may assume that $\Lambda_r$ is autoreduced for each $r\in V$.
Hence, for each $r\in V$ there must exist an autoreduced subset $\Sigma_r\subset S_r$ with $\rank\left(\Sigma_r\right)<\rank\left(\Lambda_r\right)$.
The sequence $\left(\Sigma_r:r\in V\right)$ extends to the family of components of an internal set $\Sigma\subset\ki$.
It follows that $\Sigma\subset S$ is internally autoreduced and $\rank\left(\Sigma\right)<\rank\left(\Lambda\right)$.
That is, $\Lambda$ is not minimally ranked among the internally autoreduced subsets of $S$.
This proves~(ii).
\end{proof}

By an {\em internal $\Delta$-ideal} of $\ki$ we mean an internal subset of $\ki$ almost all of whose components are $\Delta$-ideals.
Note that an internal $\Delta$-ideal is a $\Delta$-ideal of $\ki$.
Given an internal set $S\subset\ki$,  we use $\left(S\right)_{\int}$,
$\left[S\right]_{\int}$, and $\left\{ S\right\}_{\int} $ to denote the internal ideal whose $r$th component is, respectively,
the ideal, $\Delta$-ideal, and radical $\Delta$-ideal generated
by $S_r$.
On the other hand, it is not hard to check that $(S)_{\int}$ is contained in every ideal of $\ki$ containing $S$, $[S]_{\int}$ is contained in every internal $\Delta$-ideal of $\ki$ containing $S$, and $\left\{S\right\}_{\int}$ is contained in every radical internal $\Delta$-ideal of $\ki$ containing $S$.

It follows immediately from the definitions that since every $\Delta$-ideal has a characteristic subset, every internal $\Delta$-ideal has an internally characteristic subset.
Moreover, a prime internal $\Delta$-ideal is determined by an internally characteristic set in the following sense: if $P$ is a prime internal $\Delta$-ideal and $\Lambda\subset P$ is an internally characteristic set for $P$ then 
$$P=I_{\int}[\Lambda]:=\{g\in\ki:(H^{\int}_{\Lambda})^Ng\in [\Lambda]_{\operatorname{int}}, \text{ for some } N\in\mathbb N^*\}$$
where by definition $H^{\int}_{\Lambda}\in\ki$ is determined by the sequence $\left(H_{\Lambda_r}:r\in\mathbb I\right)$.
Indeed, this is just by transfer from the corresponding fact about standard prime $\Delta$-ideals.
Moreover, the characterisation of characteristic sets of prime $\Delta$-ideals (Fact~\ref{thm:Kolchin}) tranfers easily:
\begin{fact}
\label{thm:Kolchin:int}
An internal set $\Lambda\subset\ki$ is internally characteristic for a prime internal $\Delta$-ideal if and only if $\Lambda$ is internally coherent and 
$$I_{\int}(\Lambda):=\left\{g\in\ki:\left(H^{\int}_\Lambda\right)^Ng\in (\Lambda)_{\int},\text{ for some }N\in\mathbb N^*\right\}$$
is a prime internal ideal containing no nonzero elements that are internally reduced with respect to $\Lambda$.
\end{fact}
\noindent
Of course, here {\em internally coherent} means coherent in almost all components.

\bigskip
\section{Comparing $\ki$ and $\k$}
\label{k-ki-sect}

\noindent
Since $K=\prod_\mathcal{U}K_r$ is equipped with the natural (co-ordinatewise) $\Delta$-field structure we also have the $\Delta$-polynomial ring $K\{X\}$.
There is a natural embedding of $\k$ into $\ki$ over $K$ which essentially ``unpacks'' the coefficients of the polynomial.
More precisely, if we list all of the monomials in $\Theta X$ as $M_{1},M_{2},\ldots$, then we can write an element of $\k$ as $\sum_{i=1}^{\infty}a_{i}M_{i}$ with the $a_{i}\in K$, all but finitely many of which are zero.
Then we can define $\phi:\k\to\ki$ by
\[\left[\phi\left(\sum_{i=1}^{\infty}a_{i}M_{i}\right)\right](r)=\sum_{i=1}^{\infty}a_{i}(r)M_{i}\]
It is easy to check
that $\phi$ is an embedding of $\Delta$-rings that is the identity on $K$.
This allows us to view $\ki$ as a $\Delta$-ring extension of $\k$.
Following the approach of~\cite{vandenDries}, we study the extent to which the properties of $\k$ are reflected in this extension.

We have the following intrinsic characterisation of $\k$: it is the subring of elements $f\in\ki$ for which for some $V\in\mathcal U$ there is a bound on the degree and order of $f(r)$ independently of $r\in V$.
Indeed, if $f$ has this property then there are finitely many monomials $M_1,\dots,M_N$ such that for all $r\in V$ there are $a_{i,r}\in K_r$ such that $f(r)=\sum_{i=1}^Na_{i,r}M_i$, and so $f$ is $\phi\left(\sum_{i=1}^Na_iM_i\right)$ where $a_i\in K$ is such that $a_i(r)=a_{i,r}$ for all $r\in V$.
Conversely, from the definition of $\phi$ it follows that if
$f\in\k$ then there is a set $V\in\mathcal U$ such that for
all $r\in V$, $\phi(f)(r)$ has the same ``shape'' as $f$,
ie. a monomial $M$ appears with a nonzero coefficient in $f$ if
and only if it does in $\phi(f)(r)$ as well.

Let us first consider $\ki$ purely as a ring extension of $\k$.
Recall that $Z=(Z_i:i<\omega)$ is an enumeration of the algebraic indeterminates according to height.
The underlying ring structure of $\k$ is simply the pure polynomial ring $K\left[Z\right]$, while that of $\ki$ is the ring of internal polynomials $K\left[Z\right]_{\int}:=\prod_{\mathcal U}K_r\left[Z\right]$.
In the following proposition we extend some of the results of van den Dries and Schmidt -- which were proved in~\cite{vandenDries} for finitely generated polynomial rings -- to this infinitary setting.

\begin{proposition}
\label{flatness}
The ring $\ki$ is a faithfully flat $\k$-module.
Moreover, for any ideal $I$ of $\k$,
\begin{itemize}
\item[(a)]
$I$ is prime in $\k$ if and only if $I\ki$ is prime in $\ki$, and
\item[(b)]
$\big\{g\in\k:g^N\in I\ki\text{ for some }N\in\mathbb N^*\big\} = \sqrt{I}$
\end{itemize}
\end{proposition}

\begin{proof}
As above, $\k=\kz$ and $\ki=\kzi$, where $Z=(Z_i:i<\omega)$ is an enumeration of the algebraic indeterminates according to height.

For flatness consider a homogeneous linear equation
$$f_1Y_1+\cdots f_\ell Y_\ell=0$$
where $f_1,\dots,f_\ell\in\kz$.
We denote this equation by $FY=0$.
Suppose $g\in \kzi^\ell$ is a solution.
We need to show that $g$ is a $\kzi$-linear combination of solutions in $\kz^\ell$.
We do this by passing first to $\kzhi$, where $h$ is such that $f_1,\dots,f_\ell\in\kzh$, and then to $\kzh$ by the finite variable case.
For each $r\in\mathbb I$, let $h_r\geq h$ be such that $g(r)\in\krzhr^\ell$.
Since, for almost all $r$, $g(r)$ is a solution to $F(r)Y=0$, and since $\krzhr$ is flat over $\krzh$, we have that $g(r)$ is a $\krzhr$-linear combination of solutions to $F(r)Y=0$ in $\krzh^\ell$.
Now, as the coefficients of the equation are of bounded degree independently of $r$, Theorem~1.4 of~\cite{vandenDries} gives us a bound $\alpha$, independent of $r$, such that all solutions in $\krzh^\ell$ are $\krzh$-linear combinations of solutions with degree bounded by $\alpha$.
Let $N$ be the $K_r$-dimension of the subspace of $\krzh^\ell$ made up of tuples of polynomials with degree bounded by $\alpha$.
Then $N$ depends only on $\alpha$ and $h$, and hence not on $r$.
Moreover, for each $r$, there exist solutions $e_{1,r},\dots,e_{N,r}\in\krzh^\ell$ such that $g(r)$ is a $\krzhr$-linear combination of $\{e_{1,r},\dots,e_{N,r}\}$.
Setting $e_j\in\kzhi^\ell$ to be such that $e_j(r)=e_{j,r}$ almost everywhere, we have that $g$ is a $\kzi$-linear combination of $\{e_1,\dots,e_N\}$.
Theorem~1.1 of~\cite{vandenDries} tells us that $\kzhi$ is a flat extension of $\kzh$, and so each $e_j$ is a $\kzhi$-linear combination of solutions in $\kzh^\ell$. 
Hence $g$ is a $\kzi$-linear combination of solutions in $\kzh^\ell\subset\kz^\ell$, as desired.

For faithful flatness, suppose $AY=B$ is a system of $k$ linear equations in $\ell$ unknowns over $K\left[Z\right]$ with a solution in $K\left[Z\right]_{\int}^\ell$.
Let $h$ be a bound on the heights of the entries of $A$ and $B$.
Then working co-ordinatewise, and using the fact that $K_r\left[Z\right]$ is faithfully flat over $K_r[Z_i:i<h]$, we have that $AY=B$ has a solution in $K\left[Z_i:i<h\right]_{\int}$.
But the latter is faithfully flat over $K\left[Z_i:i<h\right]$ by Theorem~1.8 of~\cite{vandenDries}.
Hence we get a solution in $K\left[Z\right]^\ell$, as desired.

The right-to-left direction of part~(a) of the ``moreover'' clause is an immediate consequences of the fact that by faithfull flatness $I\ki\cap\k=I$.
For the converse, note first of all that since checking any instance of primeness of $I\ki$ involves only finitely many elements from $I$, it would suffice to prove the result with $I\cap\kzh$, which is still prime, in place of $I$, for each $h<\omega$.
But $I\cap\kzh$ is finitely generated, and hence it will suffice to prove the result for finitely generated $I$.
That is, we assume  $(\Lambda)\k$ is prime for some finite subset $\Lambda\subset\k$, and prove that $(\Lambda)\ki$ is prime.
Let $h$ be such that $\Lambda\subset\kzh$.
Since $\kz$ is faithfully flat over $\kzh$,
$(\Lambda)\kz\cap\kzh=(\Lambda)\kzh$
and hence the latter is prime.
By Theorem~2.5 of~\cite{vandenDries}, $(\Lambda)\kzhi$ is prime.
Note that $(\Lambda)\kzhi$, being finitely generated, is an internal ideal with family of components $\big((\Lambda_r)\krzh:r\in\mathbb I\big)$ where $\Lambda_r:=\{f(r):f\in\Lambda\}$.
Hence, for almost all $r$, $(\Lambda_r)\krzh$ is prime.
But then, passing to a polynomial ring extension over the same field, we have that $(\Lambda_r)K_r[Z]$ is prime for almost all $r$.
It follows that the corresponding internal ideal $(\Lambda)\kzi$ is prime, as desired.

Finally, for~(b), suppose $g\in\k$ and $N\in\mathbb N^*$ are such that $g^N\in I\ki$.
Let $\Lambda$ be a finite subset of $I$ such that $g^N\in(\Lambda)\ki$.
We show that $g\in\sqrt{(\Lambda)}$.
Let $h$ be such that $\Lambda\subset\kzh$ and $g\in\kzh$.
Then, for almost all $r$, $g(r)^{N(r)}\in(\Lambda_r)K_r[Z]\cap\krzh=(\Lambda_r)\krzh$.
That is, $g\in\sqrt[\int]{(\Lambda)\kzhi}$.
Corollary~2.7 of~\cite{vandenDries} says that $\sqrt[\int]{(\Lambda)\kzhi}=\sqrt{(\Lambda)}\cdot\kzhi$.
By the faithful flatness of $\kzhi$ over $\kzh$, we get $g\in\sqrt{(\Lambda)}$, as desired.
The reverse containment of part~(b) is clear.
\end{proof}

Now let us consider $\ki$ as a $\Delta$-ring extension of $\k$.
Note that many of the internal notions discussed in the last section restrict to the usual ones on $\k$.\label{intchar-char}
For example, suppose $f,g\in\k$.
Then $f$ is internally of lower rank than $g$ in $\ki$ (respectively internally reduced with respect to $g$) if and only if it is of lower rank than $g$ in $\k$ (respectively reduced with respect to $g$).
This is because rank and reducedness depend only on the shape of $\Delta$-polynomials, and elements of $\k$ have the same shape as their co-ordinates almost everywhere.
Similarly, a finite subset of $\k$ is autoreduced in $\k$ if and only if, as an internal subset of $\ki$, it is internally autoreduced.
Also, the ranking of internally autoreduced subsets of $\ki$ restricts to the usual ranking of autoreduced (finite) subsets of $\k$.
In particular,
{\em if $\Lambda\subset\k$ is an internally characteristic set of a prime internal $\Delta$-ideal of $\ki$, then  $\Lambda$ is a characteristic set of a prime $\Delta$-ideal of $\k$.}
Indeed, if $P\subset\ki$ is a prime internal $\Delta$-ideal witnessing the truth of the antecedent of the above statement, then $P\cap \k$ witnesses the consequent.
The converse, however, is not so straightforward, and constitutes the main goal of this section (cf. Theorem~\ref{intchar=char} below).
We will use the characterisation of characteristic subsets of prime $\Delta$-ideals given by Fact~\ref{thm:Kolchin} (and its nonstandard analogue Fact~\ref{thm:Kolchin:int}).
The following lemma is the main technical result that will allow us to do so.

\begin{lemma}
\label{extendcoherence}
Let $\Lambda \subset K\left\{ X\right\} $ be a finite set.
Then
\begin{itemize}
\item[(a)]
$\Lambda$ is coherent in $\k$ if and only if $\Lambda$ is internally coherent in $\ki$.
\item[(b)]
If $I(\Lambda)$ is prime in $\k$ then $I_{\int}(\Lambda)=I(\Lambda)\ki$, and hence $I_{\int}(\Lambda)$ is prime in $\ki$.
\item[(c)]
Suppose $\Lambda$ is autoreduced and $I(\Lambda)$ is prime in $\k$.
Then $I(\Lambda)$ contains a nonzero element that is reduced with respect to $\Lambda$ if and only if $I_{\int}(\Lambda)$ contains a nonzero element that is internally reduced with respect to $\Lambda$.
\end{itemize}
\end{lemma}

\begin{proof}
Toward a proof of part~(a), let $F$ be the set of pairs $(f,g)$ of distinct elements from $\Lambda$ such that there exists $\theta_1,\theta_2\in\Theta$ with $\theta_1v_f=\theta_2v_g$.
If $(f,g)\in F$, let $\theta_f$ and $\theta_g$ be such that $\theta_fv_{f}=\theta_gv_g$ is of minimal height, say $h(f,g)$.
Note that despite the notation both operators $\theta_f$ and $\theta_g$ depend on the pairt $(f,g)$.

Suppose that $\Lambda$ is coherent.
By the coherence of $\Lambda$, if $(f,g)\in F$ then
$S_g\left(\theta_ff\right)-S_f\left(\theta_gg\right)\in\left(\Lambda\right)_{h(f,g)-1}$.
It follows that for some $V_{f,g}\in\mathcal U$ and all $r\in V_{f,g}$,
$S_{g\left(r\right)}\left(\theta_ff\left(r\right)\right)-S_{f\left(r\right)}\left(\theta_gg\left(r\right)\right)\in\left(\Lambda_r\right)_{h(f,g)-1}$
in $K_r\left\{X\right\}$.
We are using here also the fact that the separant of a $\Delta$-polynomial can be computed co-ordinatewise almost everwhere.

Let $\displaystyle V=\bigcap_{(f,g)\in F}V_{f,g}$.
We have already pointed out that internal autoreducedness extends autoreducedness, so, shrinking $V$ if necessary, we may assume that $\Lambda_r$ is autoreduced for all $r\in V$.
Shrinking $V$ further, we may also assume that $f$ and $f(r)$ have the same shape for all $f\in\Lambda$, and that if $f\neq g$ in $\Lambda$ then $f(r)\neq g(r)$ in $\Lambda_r$.
We now show that for each $r\in V$, $\Lambda_r$ is coherent in $K_r\left\{X\right\}$.
Suppose $f(r),g(r)\in\Lambda_r$ are distinct and there are $\theta_1,\theta_2\in\Theta$ such that $\theta_1v_{f(r)}=\theta_2v_{g(r)}$ is of height $h$, and that $h$ is minimal such.
Since this is a fact just about the shape of $f(r)$ and $g(r)$, it follows that $(f,g)\in F$, $\theta_1=\theta_f$, $\theta_2=\theta_g$, and $h=h(f,g)$.
Hence, as $r\in V_{f,g}$ we have that 
$S_{g\left(r\right)}\left(\theta_1f\left(r\right)\right)-S_{f\left(r\right)}\left(\theta_2g\left(r\right)\right)\in\left(\Lambda_r\right)_{h-1}$.
This proves that $\Lambda_r$ is coherent.
Hence $\Lambda$ is internally coherent.

For the converse, suppose that $\Lambda$ is internally coherent.
Then $\Lambda$ is internally autoreduced and hence autoreduced.
Suppose $(f,g)\in F$ and set $P:=S_g\left(\theta_ff\right)-S_f\left(\theta_gg\right)$.
Let $\{Q_1,\dots,Q_\ell\}$ be $\Lambda$ together with its derivatives up to height $h(f,g)-1$.
We need to show that $P\in (Q_1,\dots,Q_\ell)\k$.
Let $V\in U$ be such that for all $r\in V$, $\Lambda_r$ is coherent, $\theta_fv_{f(r)}=\theta_gv_{g(r)}$ with height $h(f,g)$, and this height is minimal such.
Shrinking $V$ further we may also assume that $\{Q_1(r),\dots,Q_\ell(r)\}$ is $\Lambda_r$ together with all its derivatives up to height $h(f,g)-1$.
By coherence, 
$$P(r)\in\left(\Lambda_r\right)_{h(f,g)-1}=\left(Q_1(r),\dots,Q_\ell(r)\right)K_r\left\{X\right\}$$
Hence
$P\in\left(Q_1,\dots,Q_\ell\right)\ki$.
But $P$ and the $Q_j$'s are in $\k$, and, by Proposition~\ref{flatness}, $\ki$ is faithfully flat over $\k$.
Hence
$P\in(Q_1,\dots,Q_\ell)\k$, as desired.

For part~(b), recall first of all that by definition
$$I(\Lambda)=\left\{g\in K\left\{X\right\}:H_\Lambda^tg\in (\Lambda)\k,\text{ for some }t\in\mathbb N\right\}$$
and
$$I_{\int}(\Lambda)=\left\{g\in\ki:\left(H^{\int}_\Lambda\right)^Ng\in (\Lambda)_{\int},\text{ for some }N\in\mathbb N^*\right\}$$
But $H_{\Lambda}^{\int}=H_{\Lambda}$ since separants and initials of elements in $\k$ can be computed co-ordinatewise almost everywhere.
Similarly, $\left(\Lambda\right)_{\int}=\left(\Lambda\right)\ki$ because $\Lambda$ is finite and hence $\left(\Lambda\right)\ki$ is internal with components $(\Lambda_r) K_r\{X\}$ where $\Lambda_r:=\{f(r):f\in\Lambda\}$.
So in fact,
$$I_{\int}(\Lambda)=\left\{g\in\ki:H_\Lambda^Ng\in (\Lambda)\ki,\text{ for some }N\in\mathbb N^*\right\}$$
It follows that $I(\Lambda)\subset I_{\int}(\Lambda)$, and so $I(\Lambda)\ki\subseteq I_{\int}(\Lambda)$

For the reverse containment, suppose toward a contradiction that $g\in I_{\int}(\Lambda)$ and $g\notin I(\Lambda)\ki$.
Then $H_\Lambda^Ng\in (\Lambda)\ki\subset I(\Lambda)\ki$ for some $N\in\mathbb N^*$.
Since $I(\Lambda)$ is prime by assumption, Proposition~\ref{flatness}(a) tells us that $I(\Lambda)\ki$ is also prime.
Hence $H_\Lambda^N\in I(\Lambda)\ki$.
By Proposition~\ref{flatness}(b), $H_\Lambda^t\in I(\Lambda)$ for some $t\in\mathbb N$.
Increasing $t$ if necessary, $H_\Lambda^t\in (\Lambda)$.
But this means that $1\in I(\Lambda)$, contradicting the primality of that ideal.
We have shown that $I_{\int}(\Lambda)= I(\Lambda)\ki$.
In particular, by Proposition~\ref{flatness}(a), $I_{\int}(\Lambda)$ is prime.

We now turn to part~(c); $\Lambda$ is autoreduced and $I(\Lambda)$ is prime.
One direction is straightforward: a nonzero element of $I(\Lambda)$ reduced with respect to $\Lambda$ is itself a nonzero element of $I_{\int}(\Lambda)$ internally reduced with respect to $\Lambda$.
For the converse, let $Z=(Z_i:i<\omega)$ again enumerate the algebraic indeterminates according to height.
Consider the leaders $T:=\{v_f:f\in\Lambda\}$ and let $Y:=\{Z_i:Z_i\text{ is not a derivative of any }v_f\in T\}$.
Since $\Lambda$ is autoreduced, all the algebraic indeterminates appearing in $\Lambda$ come from $Y\cup T$.
It follows, as is remarked after Lemma~2 of Chapter~4 of~\cite{Kolchin}, that $I(\Lambda)$ is generated by $I:=I(\Lambda)\cap K[Y,T]$.
Hence, $I_{\int}(\Lambda)$, which by part~(b) is equal to $I(\Lambda)K[Z]_{\int}$, is in fact equal to $IK[Z]_{\int}$.

Now suppose $g\in I_{\int}(\Lambda)=IK[Z]_{\int}$ is nonzero and internally reduced with respect to $\Lambda$.
Let $V\in\mathcal U$ be such that $g(r)$ is reduced with respect to $\Lambda_r:=\{f(r):f\in\Lambda\}$ for all $r\in V$.
In particular,  $g(r)\in K_r[Y,T]$, and so $g\in K[Y,T]_{\int}\cap IK[Z]_{\int}$.
But $K[Z]_{\int}$ is faithfully flat over $K[Y,T]_{\int}$ since, for all $r$, $K_r[Z]$ is faithfully flat over $K_r[Y,T]$.
So $g\in IK[Y,T]_{\int}$.

For each $r\in\mathbb I$, set $L_r$ to be the field $K_r(Y)$ and let $L$ be the corresponding internal field.
Note that $L$ is the fraction field of $R:=K[Y]_{\int}$
The ultraproduct of the polynomial rings $L_r[T]$ gives rise to the ring of internal polynomails $L[T]_{\int}$, which extends $L[T]$ via the usual embedding.
Since $\deg_{v_f}\big(g(r)\big)<\deg_{v_f}(f)$, for all $r\in V$ and $f\in\Lambda$, we have that $g\in L[T]$.
But we saw above that $g\in IK[Y,T]_{\int}\subset IL[T]_{\int}$.
By the faithful flatness of $L[T]_{\int}$ over $L[T]$, $IL[T]_{\int}\cap L[T]=IL[T]$.
So $g\in IL[T]$.
Now we can clear denominators: there is $h\in R=K[Y]_{\int}$ such that $f:=hg\in IR[T]$.
Note that $f$ is still internally reduced with respect to $\Lambda$ as we have only multiplied by an internal polynomial in which no derivatives of $T$ appear, co-ordinatewise almost everywhere.

Write $f=h_1g_1+\cdots+h_\ell g_\ell$ where $g_1,\dots,g_\ell\in I\subset K[Y,T]$ and $h_1,\dots,h_\ell\in R[T]$.
Letting $(M_j:j\in J)$ be the monomials in $T$ that appear in $f$, the $h_i$'s, $g_i$'s, and $h_ig_i$'s, we can write
\begin{eqnarray*}
f&=&\sum_{j\in J}f_jM_j\\
h_i&=&\sum_{j\in J}h_{i,j}M_j\\
g_i&=&\sum_{j\in J}g_{i,j}M_j
\end{eqnarray*}
where $f_j,h_{i,j}\in K[Y]_{\operatorname{int}}$ and $g_{i,j}\in K[Y]$.
Let $J'\subset J$ be those indices $j$ such that $f_j\neq 0$.
Then for each $j\in J'$,
\begin{eqnarray*}
f_j&=&\sum_{i=1}^\ell\sum_{M_{j_1}M_{j_2}=M_j}h_{i,j_1}g_{i,j_2}
\end{eqnarray*}
and for each $j\in J\setminus J'$,
\begin{eqnarray*}
0&=&\sum_{i=1}^\ell\sum_{M_{j_1}M_{j_2}=M_j}h_{i,j_1}g_{i,j_2}
\end{eqnarray*}
That is, the $f_j$'s and $h_{i,j}$'s are a solution in $K[Y]_{\operatorname{int}}$ to the system of linear equations over $K[Y]$ given by
\begin{eqnarray*}
0&=&-S_j+\sum_{i=1}^\ell\sum_{M_{j_1}M_{j_2}=M_j}g_{i,j_2}S_{i,j_1}
\end{eqnarray*}
for $j\in J'$, and 
\begin{eqnarray*}
0&=&\sum_{i=1}^\ell\sum_{M_{j_1}M_{j_2}=M_j}g_{i,j_2}S_{i,j_1}
\end{eqnarray*}
for $j\in J\setminus J'$.
Note that $J'\neq\emptyset$ as $f\neq 0$.
By the flatness of $K[Y]_{\int}$ over $K[Y]$, our solution is a $K[Y]_{\operatorname{int}}$-linear combination of solutions in $K[Y]$.
In particular, there must exist a solution $(f_j',h_{i,j}')$ in $K[Y]$.
In fact, we can find such a solution such that for some $j\in J'$, $f_j'\neq 0$.
Hence
$f':=\sum_{j\in J'}f_j'M_j\neq 0$.
Moreover, setting
$h_i':=\sum_{j\in J}h_{i,j}'M_j \ \in K[Y,T]$,
we get that $f'=h_1'g_1+\cdots+h_\ell' g_\ell\in I\subset I(\Lambda)$.
Note that as $f'\in K[Y,T]$ and in each $T$-variable the degree of $f'$ is no greater than it was in $f$, $f'$ is reduced with respect to $\Lambda$.
We have found a nonzero element of $I(\Lambda)$ that is reduced with respect to $\Lambda$, as desired.
\end{proof}

We can now apply Facts~\ref{thm:Kolchin} and~\ref{thm:Kolchin:int} to see that characteristic sets for prime $\Delta$-ideals in $\k$ can be recognised in $\ki$.

\begin{theorem}
\label{intchar=char}
Suppose $\Lambda\subset\k$ is finite.
Then $\Lambda$ is a characteristic set for a prime $\Delta$-ideal of $\k$ if and only if it is an internally characteristic set for a prime internal $\Delta$-ideal of $\ki$.
\end{theorem}

\begin{proof}
The right-to-left direction we have already mentioned: If $\Lambda$ is internally characteristic for $P\subset\ki$ then it is characteristic for $P\cap\k$.
This is because internal autoreducedness and rank in $\ki$ restrict to autoreducedness and rank in $\k$.
(See the discussion on page~\pageref{intchar-char}.)

For the converse, suppose $\Lambda$ is characteristic for a prime $\Delta$-ideal $P\subset\k$.
By Fact~\ref{thm:Kolchin}, $\Lambda$ is coherent and $I(\Lambda)$ is a prime ideal of $\k$ containing no nonzero elements reduced with respect to $\Lambda$.
By Lemma~\ref{extendcoherence}, $\Lambda$ is internally coherent in $\ki$, $I_{\int}(\Lambda)$ is prime, and $I_{\int}(\Lambda)$ contains no nonzero element internally reduced with respect to $\Lambda$.
Hence, by Fact~\ref{thm:Kolchin:int}, $\Lambda$ is an internally charcateristic set for a prime internal $\Delta$-ideal.
\end{proof}

Recall that for $\Lambda\subset\k$, $I[\Lambda]=\{g\in K\{X\}:H_\Lambda^ng\in [\Lambda], \text{ for some } n\in\mathbb N\}$ and $I_{\int}[\Lambda]=\{g\in\ki:H_{\Lambda}^Ng\in [\Lambda]_{\operatorname{int}}, \text{ for some } N\in\mathbb N^*\}$.
(We are using here again the fact that, since $\Lambda\subset\k$, $H_\Lambda=H_\Lambda^{\int}$.)
One would like to prove, analogously to Lemma~\ref{extendcoherence}(b), that if $I[\Lambda]$ is prime and $\Lambda$ is its characteristic set, then $I_{\int}[\Lambda]=\big\{I[\Lambda]\big\}_{\int}$.
But we are not able to do so.
In fact, as we will see in the next section (see statement~(B) of Proposition~\ref{thm:Nonstandard_Equiv}), this would imply the definability of primality conjecture discussed in the introduction.
Nevertheless, we can use Theorem~\ref{intchar=char} to show that $I_{\int}[\Lambda]$ is least among the prime internal $\Delta$-ideals that lie above $I[\Lambda]$.

\begin{corollary}
\label{cor:Containment_Above}
Suppose $\Lambda\subset\k$ is a characteristic set for a prime $\Delta$-ideal.
Then $I_{\int}[\Lambda]$ is a prime internal $\Delta$-ideal,
$I_{\int}[\Lambda]\cap\k=I[\Lambda]$, and if $P$ is any prime internal $\Delta$-ideal of $\ki$ with $P\cap\k=I[\Lambda]$ then $I_{\int}[\Lambda]\subset P$.
\end{corollary}

\begin{proof}
By Theorem~\ref{intchar=char}, $\Lambda$ is internally characteristic for a prime internal $\Delta$-ideal of $\ki$, which must therefore be $I_{\int}[\Lambda]$.
It follows that $I_{\int}[\Lambda]\cap\k$ is prime and that $\Lambda$, being minimally ranked among the autoreduced subsets of this ideal, is characteristic for $I_{\int}[\Lambda]\cap\k$.
But  then $I_{\int}[\Lambda]\cap\k=I[\Lambda]$.

Suppose $P$ is a prime internal $\Delta$-ideal of $\ki$ with $P\cap\k=I[\Lambda]$.
Note that $H_{\Lambda}\notin P$ else it would be in $P\cap\k=I[\Lambda]$ forcing the latter to be all of $\k$.
Suppose $g\in I_{\int}[\Lambda]$ so that $H_\Lambda^Ng\in[\Lambda]_{\int}\subset P$, for some $N\in\mathbb N^*$.
Since $H_{\Lambda}\notin P$, $H_\Lambda(r)\notin P_r$ for almost all $r$, and so by primality $g(r)\in P_r$ for almost all $r$.
Hence $g\in P$.
So $I_{\int}[\Lambda]\subset P$, as desired.
\end{proof}

\begin{corollary}
\label{thm:Nonstandard_Nullstellensatz}
Suppose $S\subset\k$.
Then $\{S\}_{\int}\cap\k=\{S\}$.
\end{corollary}

\begin{proof}
If $\left\{S\right\} =\k$ then
$\{S\}_{\int}=\ki$ and we are done.
Otherwise $\{S\}$ is a proper radical $\Delta$-ideal and by the differential prime decomposition theorem
 there must exist finitely many prime $\Delta$-ideals, $P_1,\dots,P_\ell\subset\k$, such that
$\{S\}=P_1\cap\cdots\cap P_\ell$.
For each $i=1,\dots,\ell$, let $P_i=I[\Lambda_i]$ where $\Lambda_i$ is characteristic for $P_i$.
By Corollary~\ref{cor:Containment_Above}, each $I_{\int}[\Lambda_{i}]$ is a prime internal $\Delta$-ideal and $I_{\int}[\Lambda_{i}]\cap\k=P_i$.
In particular, $I_{\int}[\Lambda_1]\cap\cdots\cap I_{\int}[\Lambda_\ell]$ is a radical internal $\Delta$-ideal of $\ki$ that contains $S$.
It follows that $\{S\}_{\int}\subset I_{\int}[\Lambda_1]\cap\cdots\cap I_{\int}[\Lambda_\ell]$.
So, working in $\k$,
\begin{eqnarray*}
\{S\}
&\subset&
\{S\}_{\int}\cap\k\\
&\subset&
\big(I_{\int}[\Lambda_1]\cap\cdots\cap I_{\int}[\Lambda_\ell]\big)\cap\k\\
&=&
\big(I_{\int}[\Lambda_1]\cap\k\big)\cap\cdots\cap\big(I_{\int}[\Lambda_\ell]\cap\k\big)\\
&=&
P_1\cap\cdots\cap P_\ell\\
&=&
\{S\}
\end{eqnarray*}
which proves that $\{S\}_{\int}\cap\k=\{S\}$.
\end{proof}

\begin{remark}
We do not know if $[S]_{\int}\cap\k=[S]$ for all finite sets $S\subset\k$.
Since our proof of Corollary~\ref{thm:Nonstandard_Nullstellensatz} goes via prime $\Delta$-ideals it does not help.
One might hope to imitate what happens in the algebraic setting, where $(S)_{\int}\cap\k=(S)$ follows by faithful flatness from the fact that $(S)_{\int}=(S)\ki$.
The latter holds because if $g\in(S)_{\int}$ then for almost all $r$ we have $g(r)=\sum_{f\in S}g_{f,r}f(r)$ where the $g_{f,r}$ are in $K_r\{X\}$, and hence $g=\sum_{f\in S}g_ff$ where $g_f$ is chosen so that $g_f(r)=g_{f,r}$ almost everywhere.
But this argument cannot be extended to $\Delta$-ideals; it is not necessarily the case that $[S]_{\int}$ is the $\Delta$-ideal of $\ki$ generated by $S$.
For a counterexample consider the case when $\Delta$ is a single derivation $\delta$, $X$ is a single variable, and $\mathbb I=\omega$.
The internal $\delta$-ideal generated by the variable $X$ in $\ki$ clearly contains the internal $\delta$-polynomial $\delta^{\omega}X$ whose $r$th co-ordinate is $\delta^rX$.
But $\delta^\omega X$ is not in the $\delta$-ideal generated by $X$ in $\ki$ since the latter is the ideal generated by the set $\{\delta^mX:m<\omega\}$.
Note that this counterexample is radical (indeed prime), so that even passing to radicals does not avoid the problem.
\end{remark}

\bigskip
\section{Definability of primality and related problems}

\noindent
Throughout this section all nonstandard notions are with respect to a fixed nonprincipal ultrafilter $\mathcal U$ on $\omega$.
Also, we continue to work with $m$ commuting derivations $\Delta=\{\delta_1,\dots,\delta_m\}$ in $n$ differential indeterminates $X=(X_1,\dots,X_n)$.

The original motivation for this work was to apply nonstandard methods to the problem of the existence of uniform bounds in checking if the radical $\Delta$-ideal generated by a given finite set of $\Delta$-polynomials is prime.
This is the definability of primality problem discussed in the introduction (Conjecture~\ref{primeconjecture}).
The algebraic analogue of Conjecture~\ref{primeconjecture} was proved using nonstandard methods by van den Dries and Schmidt~\cite{vandenDries}, and we have tried here to extend their approach to the differential setting.
The following translates definability of primality into a statement about rings of internal $\Delta$-polynomials.

\begin{proposition}[Nonstandard formulation of definability of primality]
\label{primeconjecture=int}
Conjecture~\ref{primeconjecture} of the introduction is equivalent to the following statement:
\begin{itemize}
\item[(A)]
For any internal $\Delta$-field $K$ and internal $S\subset\k$, if $\{S\}_{\int}\cap\k$ is prime in $\k$ then $\{S\}_{\int}$ is prime in $\ki$.
\end{itemize}
\end{proposition}

\begin{proof}
If the conjecture fails then there exists a $d<\omega$, and for every $r<\omega$ a $\Delta$-field $K_r$ with a finite set $S_r\subset K_r\{X\}$ of order and degree $\leq d$ such that $\{S_r\}$ is not prime, but it is proper and for all $f,g\in K_r\{X\}$ of degree and order $\leq r$, if $fg\in \{S_r\}$ then either $f$ or $g$ is in $\{S\}$.
Passing to the ultraproduct we obtain an internal $\Delta$-field $K$ whose components are $K_r$ and an internal set $S\subset\ki$ whose components are $S_r$.
Since the degrees and orders of the elements of $S_r$ were bounded independently of $r$, $S$ is in fact a subset of $\k$.
Since the $\{S_r\}$ were not prime, $\{S\}_{\int}$ is not prime.
On the other hand, $\{S\}_{\int}\cap \k$ is prime.
Indeed, it is proper since each $\{S_r\}$ was proper and hence $\{S\}_{\int}$ omits $1\in\k$.
Moreover, if $f,g\in\k$ with $fg\in\{S\}_{\int}$, then for almost all $r$, $f(r),g(r)\in K_r\{X\}$ are of degree and order $\leq r$ (as they are almost everywhere equal to that of $f$ and $g$), and $f(r)g(r)\in \{S_r\}$, and hence one of $f(r)$ and $g(r)$ is in $\{S_r\}$.
Depending on which of these is favoured by the ultrafilter, we get that either $f$ or $g$ is in $\{S\}_{\int}$.
We have shown that $\{S\}_{\int}\cap\k$ is prime but $\{S\}_{\int}$ is not.
That is, statement (A) is false.

Conversely, suppose Conjecture~\ref{primeconjecture} holds, $K$ is an internal $\Delta$-field, and $S\subset\k$ is an internal set.
Let $(S_r:r<\omega)$ be a family of components for $S$.
The fact that $S\subset\k$ means that for some $V\in\mathcal U$ there is a bound on the orders and degrees of the elements of $S_r$ for all $r\in V$, that is independent of $r$.
Let $d$ be this bound and let $N=r(d,n,m)$ be given by Conjecture~\ref{primeconjecture}.
Now suppose that $\{S\}_{\int}$ is not prime.
We show that $\{S\}_{\int}\cap\k$ cannot be prime, thereby verifying the contrapositive of (A).
Shrinking $V$ if necessary, $\{S_r\}$ is not prime for all $r\in V$.
Hence there is $f_r,g_r\in K_r\{X\}$ of degree and order $\leq N$ witnessing the nonprimality of $\{S_r\}$.
(Note that while $S_r$ need not be finite, there exists finite $T_r\subset S_r$ with $\{S_r\}=\{T_r\}$ by the Ritt-Raudenbush basis theorem, and hence Conjecture~\ref{primeconjecture} applies.)
Let $f,g\in\ki$ be such that $f(r)=f_r$ and $g(r)=g_r$ almost everywhere.
Then, as the degrees and orders are bounded, we have that $f,g\in\k$.
Moreover, $f,g$ witness the nonprimality of $\{S\}_{\int}$.
That is, $\{S\}_{\int}\cap\k$ is not prime.
\end{proof}

\begin{remark}
\label{int=finite}
The above proposition remains true even if we change statement~(A) to consider only {\em finite} subsets $S\subset\k$.
Indeed, when we were proving above that a counterexample to Conjecture~\ref{primeconjecture} leads to a counterexample to~(A), we could have taken the $S_r$ to be of bounded size independently of $r$ without changing $\{S_r\}$.
This is because the $\Delta$-polynomials in $K_r\{X\}$ of order and degree bounded by $d$ form a $K_r$-vector space whose (finite) dimension depends only on $d$ (not on $r$).
Choosing the $S_r$ to be of uniformly bounded size would then have lead to a finite counterexample $S$.
\end{remark}
 
Unfortunately, the results we obtained in the last section only go part way toward proving statement (A).
The proof of the following weak form of (A) illustrates what goes wrong.

\begin{proposition}
\label{prop:Partial_Prime_Bound}
Suppose $K$ is an internal $\Delta$-field and $S\subset\k$ is such that $\{S\}_{\int}\subsetneq\ki$ satisfies: for all $f\in\k$ and $g\in\ki$, if $fg\in\{S\}_{\int}$ then $f\in\{S\}_{\int}$ or $g\in\{S\}_{\int}$.
Then $\{S\}_{\int}$ is prime.
\end{proposition}

\begin{proof}
Note that our assumption on $\{S\}_{\int}$ implies that $\{S\}_{\int}\cap \k$ is prime.\footnote{We know by Corollary~\ref{thm:Nonstandard_Nullstellensatz} that $\{S\}_{\int}\cap \k=\{S\}$, but we don't use this fact here.}
Theorem~\ref{intchar=char} tells us that if $\Lambda$ is a characteristic for $\{S\}_{\int}\cap\k$ then $I_{\int}[\Lambda]$ is prime.
We will show that $\{S\}_{\int}=I_{\int}[\Lambda]$.
Since $I_{\int}[\Lambda]$ contains $I[\Lambda]=\{S\}_{\int}\cap\k\supset S$, we have $\{S\}_{\int}\subset I_{\int}[\Lambda]$.
For the reverse containment, suppose $g\in I_{\int}[\Lambda]$.
Then $H_{\Lambda}^Ng\in [\Lambda]_{\int}\subset \{S\}_{\int}$ for some $N\in\mathbb N^*$.
It follows that $(H_{\Lambda}g)^N\in\{S\}_{\int}$.
Since $\{S\}_{\int}$ is componentwise radical almost everywhere, $H_{\Lambda}g\in \{S\}_{\int}$.
Now $H_{\Lambda}\notin\{S\}_{\int}$, else it would be in $I[\Lambda]$ and the latter would then not be proper.
Since $H_{\Lambda}\in\k$, our assumption on $\{S\}_{\int}$ implies that $g\in\{S\}_{\int}$, as desired.
\end{proof}

We leave it to the reader to deduce from the above proposition the following standard consequence which bounds one of the two universal quantifiers in the definition of primality.

\begin{theorem}
\label{prop:Partial_Prime_Bound_Standard}
For every $d$ there exists $r=r(d,n,m)$
such that for every $\Delta$-field $k$ and and every finite set of $\Delta$-polynomials $S\subset k\{ X\}$ of degree and order $\leq d$, the following are equivalent:
\begin{itemize}
\item[(i)]
$\{S\}$ is prime
\item[(ii)]
$\{S\}$ is proper and for all $f,g\in k\{X\}$ with $f$ of degree and order $\leq r$, if $fg\in \{S\}$ then either $f$ or $g$ is in $\{S\}$.
\end{itemize}
\end{theorem}

\medskip

Now we turn our attention away from trying to prove definability of primality, and instead toward proving that it is equivalent to other existence-of-bounds problems in differential algebra.
In this we are informed by~\cite{RittProbAlgorithms} where Golubitsky, Kondratieva, and Ovchinnikov show, under some natural computability-theoretic assumptions on the base $\Delta$-field\footnote{Namely that $k$ is a computable $\Delta$-field and that there is an algorithm for determining if a univariate polynomial over $k$ is irreducible.}, the equivalence of several computational problems in differential algebra.
We are interested in the existence-of-bounds analogues of some of their existence-of-algorithm problems.

\begin{proposition}
\label{thm:Nonstandard_Equiv}
Statement~{\em (A)} of Proposition~\ref{primeconjecture=int} is equivalent to each of the following statements: For any internal $\Delta$-field $K$,
\begin{itemize}
\item[(B)]
If $P\subset\ki$ is a prime internal $\Delta$-ideal with an internally characteristic set from $\k$ then $P=\{S\}_{\int}$ for some internal $S\subset\k$.
\item[(C)]
Suppose $P,Q\subset\ki$ are prime internal $\Delta$-ideals with internally characteristic sets from $\k$.
If $P\cap\k\subset Q$ then $P\subset Q$.
\item[(D)]
Suppose $S\subset\k$ is internal and $P, Q\subset\ki$ are minimal prime internal $\Delta$-ideals containing $S$.
If $P\cap\k\subset Q$ then $P=Q$.
\item[(E)]
If $S\subset\k$ is internal and $f\in\k$ is not a zero divisor in $\k/\{S\}$ then $f$ is not a zero divisor in $\ki/\{S\}_{\int}$.
\end{itemize}
\end{proposition}

\begin{proof}
Suppose (A) holds and let $\Lambda\subset\k$ be internally characteristic for a prime internal $\Delta$-ideal $P$.
Then $P=I_{\int}[\Lambda]$, and by Theorem~\ref{intchar=char}, $\Lambda$ is characteristic for the prime $\Delta$-ideal $I[\Lambda]$ of $\k$.
Let $S\subset\k$ be finite so that $I[\Lambda]=\{S\}$.
By Corollary~\ref{cor:Containment_Above}, $P$ is least among the prime internal $\Delta$-ideals that lie above $\{S\}$.
By Corollary~\ref{thm:Nonstandard_Nullstellensatz}, $\{S\}_{\int}\cap\k=\{S\}$, and by~(A), $\{S\}_{\int}$ is prime.
Hence $P\subset\{S\}_{\int}$.
The reverse containment is clear.
This proves~(B).

That~(B) implies~(C) is immediate.

Note that~(D) does not follow tautologically from~(C), because the prime internal $\Delta$-ideals $P$ and $Q$ appearing in~(D) are not assumed to have internally characteristic sets from $\k$.
However, by the following lemma, $P$ and $Q$ {\em do} in fact have internally characteristic sets from $\k$, and so~(C) does imply~(D).

\begin{lemma}
\label{minintchar}
If $S\subset\k$ and $P\subset\ki$ is a minimal prime internal $\Delta$-ideal containing $S$, then $P$ has an internally characteristic set from $\k$.
\end{lemma}
\begin{proof}[Proof of Lemma~\ref{minintchar}]
Consider the prime $\Delta$-ideal $Q:=P\cap\k$, and let $\Lambda\subset\k$ be a characteristic set for $Q$.
By Theorem~\ref{intchar=char}, $\Lambda$ is internally characteristic for the prime internal  $\Delta$-ideal $I_{\int}[\Lambda]$.
We show that $P=I_{\int}[\Lambda]$.
By Corollary~\ref{cor:Containment_Above}, $I_{\int}[\Lambda]$ is least among the prime internal $\Delta$-ideals of $\ki$ whose intersection with $\k$ is $Q$.
Hence $I_{\int}[\Lambda]\subset P$.
On the other hand, $S\subset P\cap\k=Q\subset I_{\int}[\Lambda]$.
So by the minimality of $P$, $P=I_{\int}[\Lambda]$, as desired.
\end{proof}

In order to prove that~(D) implies~(E), we first show, unconditionally, that if $S\subset\k$ is internal and $\Lambda\subset\k$ is characteristic for a minimal prime $\Delta$-ideal containing $S$, then $I_{\int}[\Lambda]$ is a minimal prime internal $\Delta$-ideal of $\ki$ containing $S$.
Indeed, by Theorem~\ref{intchar=char}, $I_{\int}[\Lambda]$ is a prime internal $\Delta$-ideal containing $S$, and hence by transfer from the standard setting, there exists $P\subset I_{\int}[\Lambda]$ minimal such.
But then, by Lemma~\ref{minintchar}, $P=I_{\int}[\Sigma]$ where $\Sigma\subset\k$ is internally characteristic for $P$.
Intersecting with $\k$, we get $S\subset I[\Sigma]\subset I[\Lambda]$, and hence $I[\Sigma]=I[\Lambda]$ by minimality.
Therefore $P=I_{\int}[\Sigma]=I_{\int}[\Lambda]$, as desired.

Now we show that~(D) implies:
\begin{itemize}
\item[(D$^\prime$)]
If $Q\subset\ki$ is a minimal prime internal $\Delta$-ideal containing $S$, then $Q\cap\k$ is a minimal prime $\Delta$-ideal of $\k$ containing $S$.
\end{itemize}
Indeed, let $Q'\subset Q\cap\k$ be a minimal prime $\Delta$-ideal containing $S$.
Let $\Lambda\subset\k$ be characteristic for $Q'$, so that $Q'=I[\Lambda]$.
Then, by the previous paragraph, $P:=I_{\int}[\Lambda]$ is another minimal prime internal $\Delta$-ideal of $\ki$ containing $S$.
As $P\cap\k=I[\Lambda]\subset Q$, (D) implies that $P=Q$.
So, $Q\cap\k=I[\Lambda]=Q'$

We are ready to prove that~(D) implies~(E).
We use the following characterisation of zero divisors that comes from the prime decomposition theorem (see, for example, Lemma~1 of~\cite{RittProbAlgorithms}): $f\in\k$ is a zero divisor modulo $\{S\}$ if an only if $f$ belongs to one of, but not all of, the minimal prime $\Delta$-ideals of $\k$ that contain $S$.
Working componentwise almost everywhere we also have: $f\in\ki$ is a zero divisor in $\ki/\{S\}_{\int}$ if an only if $f$ belongs to one of, but not all of, the minimal prime internal $\Delta$-ideals of $\ki$ that contain $S$.
Now, toward a proof of the contrapositive of~(E), suppose that $f\in\k$ is a zero divisor in $\ki/\{S\}_{\int}$, and let $P$ and $Q$ be minimal prime internal $\Delta$-ideals containing $S$ such that $f\in P\setminus Q$.
But then $f\in\big(P\cap\k\big)\setminus\big(Q\cap\k\big)$.
As $P\cap\k$ and $Q\cap\k$ are minimal prime internal $\Delta$-ideals containing $S$ by~(D$'$), it follows that $f$ is a zero divisor of $\k/\{S\}$.

Finally, we show that~(E) implies~(A).
Assuming~(E), suppose that $S\subset\k$ is an internal set such that $\{S\}_{\int}$ is not prime in $\ki$.
We need to prove that $\{S\}_{\int}\cap\k$ is not prime in $\k$.
By Corollary~\ref{thm:Nonstandard_Nullstellensatz}, $\{S\}_{\int}\cap\k=\{S\}$.
Now by Proposition~\ref{prop:Partial_Prime_Bound}, the fact that $\{S\}_{\int}$ is not prime is witnessed by some $f\in\k$ and $g\in\ki$ neither of which are in $\{S\}_{\int}$ but such that $fg\in\{S\}_{\int}$.
That is, $f$ is a zero divisor in $\ki/\{S\}_{\int}$.
As $f\in \k$, (E) implies that $f$ is a zero divisor in $\k/\{S\}$.
Hence $\{S\}$ is not prime in $\k$, as desired.
\end{proof}

\begin{theorem}
\label{primeconjecture-equivalents}
The following statements in differential algebra are equivalent:
\begin{itemize}
\item[(1)]
{\em Definability of primality.}\\
Conjecture~\ref{primeconjecture}.
\item[(2)]
{\em Bounds for generators of a prime $\Delta$-ideal given a characteristic set.}\\
For every $d$ there exists $r=r(d,n,m)$ such that for every $\Delta$-field $k$, if $P$ is a prime $\Delta$-ideal of $k\{X\}$ with a characteristic set whose elements are of degree and order $\leq d$, then $P$ is radically differentially generated by $\Delta$-polynomials of order and degree $\leq r$.
\item[(3)]
{\em Bounds for checking $\subset$ among prime $\Delta$-ideals given characteristic sets.}\\
For every $d$ there exists $r=r(d,n,m)$ such that for every $\Delta$-field $k$ and all prime $\Delta$-ideals $P,Q\subset k\{X\}$ with characteristic sets whose elements are of degree and order $\leq d$, if every $\Delta$-polynomial in $P$ of degree and order $\leq r$ is in $Q$ then $P\subset Q$.
\item[(4)]
{\em Bounds for distinguishing minimal prime $\Delta$-ideals.}\\
For every $d$ there exists $r=r(d,n,m)$ such that for every $\Delta$-field $k$, every set $S\subset k\{X\}$ of $\Delta$-polynomials of degree and order $\leq d$, and every pair $P$ and $Q$ of minimal prime $\Delta$-ideals containing $S$, if every $\Delta$-polynomial in $P$ of degree and order $\leq r$ is in $Q$ then $P=Q$.
\item[(5)]
{\em Definability of zero-divisibility.}\\
For every $d$ there exists $r=r(d,n,m)$ such that for every $\Delta$-field $k$, every set $S\subset k\{X\}$ of $\Delta$-polynomials of degree and order $\leq d$, and every $g\in k\{X\}$ of degree and order $\leq d$, if $gf\notin\{S\}$ for all $f\notin\{S\}$ of degree and order $\leq r$, then $g$ is not a zero divisor modulo $\{S\}$.
\end{itemize}
\end{theorem}

\begin{proof}
These statements are just standard formulations of the corresponding statements~(A) through~(E) of Proposition~\ref{thm:Nonstandard_Equiv}.
The equivalence of~(A) and~(1) is Proposition~\ref{primeconjecture=int}, and very similar arguments yield the equivalence of~(B) to~(2), (C) to~(3), (D) to~(4), and (E) to~(5).
The theorem then follows immediately from Proposition~\ref{thm:Nonstandard_Equiv}.
\end{proof}

\begin{remark}
The existence-of-algorithms analogue of~(3) would say that there is an algorithm for deciding whether one prime $\Delta$-ideal is contained in another given their characteristic sets.
This is Kolchin's Problem~3 of $\S$IV.9 of~\cite{Kolchin}, a solution to which would solve the Ritt problem of computing the prime components of the radical $\Delta$-ideal generated by a given finite set of $\Delta$-polynomials.
\end{remark}

\bigskip
\section{Two standard consequences}

\noindent
We conclude by extracting the existence of some bounds in the theory of $\Delta$-polynomial rings from our understanding of the ring of internal $\Delta$-polynomials.

Throughout  $\Delta=\{\delta_1,\dots,\delta_m\}$ and $X=(X_1,\dots,X_n)$.

\medskip
\subsection{Characteristic sets for minimal prime $\Delta$-ideals.}
Given a finite set of $\Delta$-polynomials $S$, consider the problem of finding the minimal prime $\Delta$-ideals that contain $S$.
In constructive differential algebra there has been significant work on bounds in differential elimination algorithms from which it seems that explicit bounds on characteristic sets for the minimal prime $\Delta$-ideals containing $S$ should be deducible.
We have in mind~\cite{rgalg} and~\cite{DiffNull}.
Here we show the {\em existence} of such bounds as an immediate consequence of our study of $\ki$.

\begin{theorem}
\label{pro:CharSetBound}
Suppose $d\in\mathbb{N}$.
There
is $r=r(d,n,m)\in\mathbb{N}$ such that for every
$\Delta$-field $k$ and every finite set of $\Delta$-polynomials $S\subset k\{ X\}$ of degree and order $\leq d$, each minimal
prime $\Delta$-ideal containing $S$ has a characteristic set of $\Delta$-polynomials of degree and order $\leq r$.
\end{theorem}

\begin{proof}
Suppose not.
Then for each $r\in\mathbb{N}$ there is a $\Delta$-field $K_r$, a finite set $S_r\subset K_r\{X\}$ of $\Delta$-polynomials of degree and order bounded by $d$, and a  minimal prime $\Delta$-ideal $P_r\subset K_r\{X\}$ containing $S_r$ all of whose characteristic sets contain at least one element of either degree or order strictly bigger than $r$.

Fix a non-principal ultrafilter $\mathcal U$ on $\omega$, and work in the ultraproduct $\ki$.
Let $S\subset\ki$ be the internal set with components $S_r$, and let $P\subset\ki$ be the prime internal $\Delta$-ideal with components $P_r$.
Then $P$ is a minimal prime internal $\Delta$-ideal containing $S$.
As the degrees and orders of the elements of $S_r$ were bounded independently of $r$, we actually have that $S\subset\k$.
By Lemma~\ref{minintchar}, $P$ has an internally characteristic set from $\k$, say $\Lambda\subset\k$.
It follows that for almost all $r$, $\Lambda_r:=\{f(r):f\in\Lambda\}$ is characteristic for $P_r$.
Since for $f\in\k$, $f$ and $f(r)$ have the same shape, if $t$ is the maximum of the degrees and orders of the elements of $\Lambda$, then $t$ is also the maximum of the degrees and orders of the elements of $\Lambda_r$.
Letting $r\geq t$, this contradicts the fact that $P_r$ has no characteristic set bounded by $r$  in this way.
\end{proof}

When $\{S\}$ itself is prime we get from the above theorem the existence of a bound for a characteristic set of $\{S\}$.
In this case an explicit bound (at least for orders) is given by Lemma~14 of~\cite{DiffNull}.

\begin{corollary}
There
is $r=r(d,n,m)\in\mathbb{N}$ such that for every
$\Delta$-field $k$ and every prime $\Delta$-ideal $P=\{S\}$, where $S$ is a finite set of $\Delta$-polynomials of degree and order $\leq d$,
$P$ has a characteristic set of $\Delta$-polynomials of degree and order $\leq r$.
\end{corollary}

\medskip
\subsection{Differential Nullstellensatz}
Consider finally the problem of determining if a $\Delta$-polynomial $g$ is in the radical $\Delta$-ideal generated by a finite set of $\Delta$-polynomials~$S$.
An explicit bound on how many derivatives one has to apply to $S$ to witness this membership, depending only on the orders and degrees of the elements of $S\cup\{g\}$, was found in \cite{DiffNull}.
Again, the mere {\em existence} of such a bound can be seen as an immediate corollary of our study of $\ki$, as we now point out.
Another proof of the existence of a bound, also using ultraproducts (though of differentially closed fields), was given by Singer (cf. $\S 6$ of~\cite{DiffNull}).
However Singer's bound depends also on the size of $S$.

\begin{theorem}
\label{thm:Effective_Nullstellensatz}
Suppose $d\in\mathbb N$.
There
is $r=r(d,n,m)\in\mathbb{N}$ such that for every
$\Delta$-field $k$ and every finite set of $\Delta$-polynomials $S\subset k\{ X\}$ of degree and order $\leq d$, the following holds:

if $g\in \{S\}$ is of order and degree $\leq d$
then $g^r\in\big(\theta f\mid f\in S,\theta\in\Theta,\ord\theta\leq r\big)$.
\end{theorem}

\begin{proof}
Suppose not.
Then for each $r\in\mathbb{N}$ there is a $\Delta$-field $K_r$, a finite set $S_r\subset K_r\{X\}$, and $g_r\in\{S_r\}$, such that every element of $S_r\cup\{g_r\}$ is of degree and order bounded by $d$, and
$g_r^r\notin\big(\theta f\mid f\in S_r, \theta\in\Theta, \ord\theta\leq r\big)$.
Let $S\subset\ki$ be the internal set with components $S_r$, and $g\in\ki$ the internal $\Delta$-polynomial with co-ordinates $g_r$.
Then $g\in\{S\}_{\operatorname{int}}$.
Since the order and degrees of the components were bounded, we have $S\subset\k$ and $g\in\k$.
Hence $g\in\{S\}_{\operatorname{int}}\cap\k=\{S\}$ by Corollary~\ref{thm:Nonstandard_Nullstellensatz}.
It follows that for some $f_1,\dots,f_\ell\in S$ and some $t\in\mathbb N$,
$g^t\in\big(\theta f_i\mid1\leq i\leq \ell, \theta\in\Theta, \ord(\theta)\leq t\big)\k$.
So for almost all $r\in\mathbb N$,
$g_r^t\in\big(\theta f_i(r)\mid1\leq i\leq \ell, \theta\in\Theta, \ord(\theta)\leq t\big)K_r\{X\}$.
Since $f_i(r)\in S_r$,  we have $g_r^t\in\big(\theta f\mid f\in S_r, \theta\in\Theta, \ord\theta\leq t\big)$, for almost all $r$.
But taking $r\geq t$, this contradicts the fact that $g_r^r\notin\big(\theta f\mid f\in S_r, \theta\in\Theta, \ord\theta\leq r\big)K_r\{X\}$.
\end{proof}


\begin{thebibliography}{1}

\bibitem{rgalg}
O.~Golubitsky, M.~Kondratieva, M.~Moreno~Maza, and A.~Ovchinnikov.
\newblock A bound for {R}osenfeld-{G}roebner algorithm.
\newblock {\em Journal of Symbolic Computation}, 43(8):582--610, 2008.

\bibitem{RittProbAlgorithms}
O.~Golubitsky, M.~Kondratieva, and A.~Ovchinnikov.
\newblock On the generalized {R}itt problem as a computational problem.
\newblock {\em Fundamental\cprime naya i Prikladnaya Matematika},
  14(4):109--120, 2008.

\bibitem{DiffNull}
O.~Golubitsky, M.~Kondratieva, A.~Ovchinnikov, and A.~Szanto.
\newblock A bound for orders in differential {N}ullstellensatz.
\newblock {\em Journal of Algebra}, 322(11):3852--3877, 2009.

\bibitem{hermann26}
G.~Hermann.
\newblock Die {F}rage der endlich vielen {S}chritte in der {T}heorie der
  {P}olynomideale.
\newblock {\em Mathematische Annalen}, 95(1):736--788, 1926.

\bibitem{Kolchin}
E.~R. Kolchin.
\newblock {\em Differential algebra and algebraic groups}.
\newblock Academic Press, New York, 1973.
\newblock Pure and Applied Mathematics, Vol. 54.

\bibitem{seidenberg74}
A.~Seidenberg.
\newblock Constructions in algebra.
\newblock {\em Transactions of the American Mathematical Society},
  197:273--313, 1974.

\bibitem{vandenDries}
L.~van~den Dries and K.~Schmidt.
\newblock Bounds in the theory of polynomial rings over fields. {A} nonstandard
  approach.
\newblock {\em Inventiones Mathematicae}, 76:77--91, 1984.

\bibitem{NonstandardReference}
L.~van~den Dries and A.~J. Wilkie.
\newblock Gromov's theorem on groups of polynomial growth and elementary logic.
\newblock {\em Journal of Algebra}, 89(2):349--374, 1984.

\end{thebibliography}

\def\cprime{$'$}

\end{document}